\documentclass[12pt]{article}

      \usepackage{latexsym}
         \usepackage[reqno, namelimits, sumlimits]{amsmath}
         \usepackage{amssymb, amsfonts}
         \usepackage{amsthm}

\newtheorem{Theorem}{Theorem}[section]
\newtheorem{Lemma}{Lemma}[section]
\newtheorem{Corollary}{Corollary}[section]
\newtheorem{Proposition}{Proposition}[section]

\newcommand{\be}{\begin{equation}}
\newcommand{\ee}{\end{equation}}

\newcommand{\rth}{R^3}

\renewcommand{\div}{\operatorname{div}}

\newcommand{\pd}[1]{{\frac{\partial}{\partial x_{#1}}}}

\newcommand{\irth}{\int_{R^3}}
\newcommand{\izt}{\int_0^t}

\newcommand{\lthx}[2]{L_t^{#1}\dot H_x^{#2}}
\newcommand{\hhd}{\dot H^{1/2}}
\newcommand{\hds}[1]{\dot H^{#1}}
\newcommand{\ltx}[2]{L_t^{#1} L_x^{#2}}

\newcommand{\ve}{\varepsilon}

\renewcommand{\ll}{\lambda}

\newcommand{\qzr}{Q_{{z_0},r}}
\newcommand{\bxr}{B_{{x_0},r}}
\newcommand{\ener}{\ltx \infty 2 \cap \lthx 2 1}
\newcommand{\tmax}{T_{\rm max}}
\newcommand{\tmaxu}{\tmax(u_0)}
\newcommand{\tqr}{\tilde Q_{x_0,r}}

\newcommand{\BB}{\mathcal B}
\newcommand{\uo}{u_0}
\newcommand{\rmax}{\rho_{\rm max}}
\newcommand{\OO}{\mathcal O}
\newcommand{\EE}{\mathcal E}
\newcommand{\MM}{\mathcal M}

\newcommand{\dint}{\int\!\!\!\int}

\newcommand{\R}{{\mathbb R}}

\title{
Minimal initial data for potential Navier-Stokes singularities
 }

 \author{W. Rusin\footnote{University of Minnesota. Supported in part by the Graduate School Dissertation Fellowship}\and
   V. \v Sver\'ak\footnote{University of Minnesota. Supported in part by NSF
  Grant DMS-0800908}}

\date{ }
\begin{document}

\numberwithin{equation}{section}

\maketitle

\begin{abstract} Assuming some  initial data $u_0\in\dot H^{1/2}(R^3)$ lead to a singularity  for the 
3d Navier-Stokes equations, we show that there are also initial data with the minimal 
$\dot H^{1/2}$ -- norm which will produce a singularity.
\end{abstract}

\section{Introduction}\label{sect1}

We consider the Cauchy problem for the Navier-Stokes equations in $\rth\times(0,\infty)$
\begin{eqnarray}
\left.
\begin{array}{rcl}
u_t+ u\nabla u + \nabla p -\Delta u & = & 0 \label{nse}\\
\div u & = & 0
\end{array}
\right\}&  &
\mbox{in $R^3\times(0,\infty)$} \\
\begin{array}{ccl}
u(\,\cdot\,,0) & = & u_0\,\,\,\,\,\,\,\,\,\,\,\,
\end{array} & & \, \mbox{in $R^3$} \label{inc}
\end{eqnarray}

In this paper  we will be interested in the case when the initial condition $u_0$ belongs to
the space $\hhd(\rth)$. The $\hhd$-norm is invariant under the natural scaling of the initial
data $u_0(x) \to \ll u_0(\ll x)$, and the Cauchy problem is known to be globally well-posed
for sufficiently small $u_0\in\hhd$, and locally well-posed for any $u_0\in\hhd$,
as proved by Fujita and Kato \cite{KatoH}.
These statements have to be made more precise
by specifying the exact notion of the solution. The solutions constructed by Kato
are usually called the mild solutions. See Section~\ref{sec:mildsolutions} for details.
For $\uo\in\hhd$ we denote by $T_{\rm max}(\uo)$ the maximal
time of existence of the mild solution starting at $\uo$.
Let $\BB_\rho=\{u_0\in\hhd, ||u_0||_{\hhd}<\rho\}$, and let $\rmax$ be the supremum of all $\rho>0$
such that the Cauchy problem~\ref{nse}--~\ref{inc} is globally well-posed for $u_0\in\BB_\rho$.

It is not known if $\rmax$ is finite or infinite. Here we will be interested in the hypothetical
situation when $\rmax$ is finite. In principle $\rmax$ could be finite for various reasons, which
depend on the exact notion of the solution.
However, one can show that with the natural definition of the mild solution,
the only reason $\rmax$ could be finite is the appearance of finite-time singularities
in the solution $u$ for some initial data $\uo$.\footnote{The proof of the statement uses some special properties
of the system~\ref{nse}--~\ref{inc}, and can fail for other equations with similar non-linearities covered
by the same perturbation theory, such as the complex viscous Burgers equation.
In particular, the energy inequality plays an important role in the proof.}
 We will consider the following question, motivated
by a discussion of one of the authors with Isabelle Gallagher:

\medskip
\noindent
(Q) {\sl If $\rmax$ is finite, does there exist an initial datum $u_0\in\hhd$ with
$||u_0||_{\hhd}=\rmax$, such that the solution $u$ of the Cauchy problem~\ref{nse}--~\ref{inc}
develops a singularity in finite time?}

\medskip

We will show that the answer to the question is affirmative, see Corollary~\ref{Mcompactness}.

The initial data $u_0$ with $||\uo||_{\hhd}=\rmax$ leading to a singularity
will be called {\it $\hhd$-minimal singularity-generating data}.
We will show that, if singularities exist, the set of the $\hhd$-minimal
singularity-generating data is in fact a (nonempty) subset of $\hhd$ which
is compact, modulo the action of the scalings $u_0(x)\to \ll \uo(\ll x)$
and translations $\uo(x)\to \uo(x-x_0)$.

Corollary~\ref{Mcompactness} is an easy consequence of Theorem~\ref{main}, and Lemma~\ref{stability},
 which are in some
sense the main new observation of this paper. Theorem~\ref{main} says, roughly speaking, that the
solutions of the Cauchy problem are stable with respect to the weak convergence
of the initial data.\footnote{There are several definitions of solutions and
therefore one has to formulate the result with some care - see Section~\ref{sec:leray}
for details.} This question was studied by I. Gallagher in \cite{Gallagher} and
Theorem~\ref{main} can be thought of as a continuation of those studies.

Lemma~\ref{stability} says, roughly speaking, that singularities are stable under
weak convergence of the solution in the energy norm.

Our results can also be used to show that the absence of singularities in all
(reasonable) solutions is equivalent to certain a-priori estimates.
Such statements were already proved in \cite{Gallagher} and \cite{Tao},
and we give another illustration of this principle.

Throughout this paper our main space for the initial data is the space $\hhd$,
which is the unique $\dot H^s$ space invariant under the natural scaling
of the equation. It is natural to ask if our results are true for other
scale-invariant spaces, such as $L^3$, the Morrey space $M$ with the norm
with the norm $||u||_M^2=\sup_{x,r} r^{-1}\int_{B_{x,r}}|u|^2$ studied
in~\cite{Taylor}, or some other suitable spaces covered by~\cite{KochTataru}.
We plan to address these questions in the future.

In the case of critical dispersive equation, the notion of minimal blow-up
solutions (with a definition quite different from ours) and related
profile decomposition has played an important role in the recent remarkable
advances, see for example \cite{Bourgain, CKSTT, KenigMerle, BahouriGerard}.
These techniques have been recently also applied to the Navier-Stokes
regularity in critical spaces, see \cite{KenigKoch}.

The situation considered here is different, in that we focus only on the initial
data, since we do not have bounds in critical norms for  general solutions.

\noindent
\section{Suitable weak solutions}\label{preliminaries}
We first define {\it suitable weak solutions} of the Navier-Stokes equations, as introduced
by \cite{CKN}. See also \cite{Lin} and \cite{Scheffer}. This is a local
notion. Let $\OO$ be an open subset of the space-time $\rth\times R$ and let
$u=u(x,t)=(u_1(x,t),u_2(x,t),u_3(x,t)), \,\,p=p(x,t)$ be functions in $\OO$
such that
\begin{itemize}

\item $u$ belongs locally to the {\it energy space} $\ltx \infty 2 \cap \lthx 2 1$\,,

\item $p$ belongs locally to the space $\ltx {3/2} {3/2}$\,,

\item the equations $\div u = 0 $ and $u_t+\div(u\otimes u) + \nabla p - \Delta u = 0$ are
     satisfied on $\OO$ in the sense
      of distributions\,, and

\item the local energy inequality
      \be
      \label{localenergy}
      2\dint |\nabla u|^2\phi\,\,dx\,dt \le \dint\left[ |u|^2(\phi_t+\Delta u) + (|u|^2+2p)u\nabla\phi\right]\,dx\,dt
       \ee
       is satisfied for every non-negative smooth  test function $\phi=\phi(x,t)$
       compactly supported in $\OO$.

\end{itemize}

In what follows we will use standard notation for euclidean balls centered at $x_0\in R^n$   and parabolic balls
$Q_{z_0,r}$ centered at $z_0=(x_0,t_0)\in R^n\times R$:
\begin{eqnarray*}
B_{x_0,r} & = & \{x\in R^n;\,\,|x-x_0| < r\} \\
Q_{z_0,r} & = & B_{x_0,r}\times (t_0-r^2,t_0]
\end{eqnarray*}

Given a suitable weak solution $(u,p)\,$, a point $z_0=(x_0,t_0)\in\OO$ is called a {\it regular point}
of $(u,p)$ if $u$ is H\"older continuous in a neighborhood of $z_0$.
A {\it singular point} $z_0\in\OO$ of $(u,p)$ is any point which is not regular.
We will use the following
two propositions, the various versions of which can be found in \cite{CKN, Lin, Scheffer, LadSer}.
The version below contains some quantitative estimates which are often
not explicitly stated in the literature, although they are implicit in the
proofs. A sketch of the proof of the spatial derivatives estimates
can be found for example in \cite{NRS}.
\begin{Proposition}{\bf ($\ve$-regularity criterion)}
\label{epsiloncriterion}
There exists $\ve_0>0$ such that the following statement is true:

\noindent
If $(u,p)$ is a suitable weak solution in $\OO$, such that
\be
\label{smallness}
{1 \over r^2} \dint_{Q_{z_0,r}}(|u|^3+|p|^{3/2})\,dx\,dt < \ve_0\,\,,
\ee
for some $Q_{z_0,r}$ compactly contained in $\OO$, then all points
in $Q_{z_0,r/2}$ are regular points of $(u,p)$.
Moreover, in $Q_{z_0,r/2}$ one has
\begin{eqnarray}
\label{derivativeestimates}
|\nabla^k u| & \le &  C_k r^{-1-k}\qquad k=0,1,\dots\qquad \mbox{and} \\
|u(x,t)-u(x,t')| & \le &  C'|t-t'|^{1/3}
\end{eqnarray}
\end{Proposition}

\noindent
Remark: The regularity in $t$ cannot be improved, due to solutions
of the form $u(x,t)=\nabla h(x,t)$ with $h$ harmonic in $x$ and having arbitrary
dependence on $t$. The H\"older exponent in $t$ for these solutions
is dictated by the assumptions on the integrability of the pressure
$p=-|\nabla h|^2/2 - h_t$, and  under the assumptions of the lemma
the H\"older exponent $1/3$ is optimal.

\begin{Proposition}{\bf (Compactness)}
\label{compactness}
Let $(u^k,p^k)\,\,, k=1,2,\dots$ be a sequence of suitable weak solutions such that
$u^k$ are uniformly bounded in the energy space $\ener$
on compact subsets of $\OO$ and $p^k$ are  uniformly bounded in $\ltx {3/2} {3/2}$
on compact subsets of $\OO$. Then the sequence $u^k$ is compact in $\ltx 3 3$
on compact subsets of $\OO$. Moreover, if $u^k\to u$ in $\ltx 3 3$ on compact subsets
of $\OO$ and $p^k\rightharpoonup p$ in $\ltx {3/2} {3/2}$ on compact
subsets of $\OO$, then $(u,p)$ is again a suitable weak solution.
\end{Proposition}

The two previous propositions imply the following lemma, which will
be important for the proof of our main results.

\begin{Lemma}{\bf (Stability of singularities)}
\label{stability}
In the situation of Proposition~\ref{compactness}, assume
that $z^k\in\OO$ are singular points of $(u^k,p^k)\,,\,k=1,2,\dots$,
and that $z^k\to z_0\in\OO$.
Then $z_0$ is a singular point of $(u,p)$.
\end{Lemma}

\begin{proof}
If the regularity criterion in Proposition~\ref{epsiloncriterion} did not contain
the pressure $p$, the statement of the Lemma would be immediate: indeed,
if $z_0$ is a regular point of $u$, then
$r^{-2}\dint_{\qzr}|u|^3\,\,dx\,dt = O(r^3)$ as $r\to 0_+$.
Choosing a sufficiently small $r$,  one sees that
$r^{-2}\dint_{\qzr} |u^k|^3\,\,dx\,dt$ is small for large $k$
by the strong convergence of $u^k$ in $L_{t,x}^3$.
However, such argument cannot be applied to the pressure term,
since the sequence $p^k$ may not have a subsequence which is
compact in $L_{t,x}^{3/2}$.
It is well known how to deal with this difficulty: the trick
can be found in one form or another in the proofs of partial
regularity \cite{CKN, Lin, Scheffer, LadSer}.
The pressure $p^k$ solves the equation
\be
\label{pressureeq}
-\Delta p^k = \partial_i\partial_j(u^k_i u^k_j)
\ee
Recall that the term $ u^k_i u^k_j$ is compact in $L^{3/2}_{t,x}$
on compact subsets of $\OO$. Therefore we can invert the Laplacian
in~\ref{pressureeq} using a suitable boundary condition,
(or just taking the Riesz transforms $\tilde p^k=R_iR_j(u^k_i u^k_j\chi_{B_{x_0,r}}$)
and decompose $p^k$ as
\be
\label{pressuredec}
p^k=\tilde p^k + h^k
\ee
with $\tilde p^k$ compact in $L^{3/2}_{t,x}(\qzr)$
(by Calderon-Zygmund estimates)
and $h^k$ bounded in $L^{3/2}_{t,x}(\qzr)$ and harmonic
in $x$ in $\qzr$. The term with $\tilde p^k$ can be dealt with
in the same way as the term with $u^k$.
The term $h^k$ is handled by using classical estimates
for harmonic functions:
let $\gamma\ge 1$ and let $h\in L_x^{\gamma}(\bxr)$ be harmonic in
$\bxr$. We denote
$(h)_{r'}=|B_{x_0,r'}|^{-1}\int_{B_{x_0,r'}}h$.
For $r'\le r/2$ and $x'\in B_{x_0,r'}$ we have
\be
\label{hdecay}
|h(x')-(h)_{r'}|^\gamma
\le C_{\gamma}{\left({{r'}\over r}\right)^\gamma}r^{-3}\int_{B_{x_0,r}}|h|^\gamma \,dx
\ee
We recall that we can change the pressure by any function depending on $t$ only.
Therefore we can use~\ref{hdecay}  with $h=h^k$, and integrating over $Q_{z_0,r'}$, we get
the required smallness of the term $(r')^{-2} \dint_{Q_{z_0,r'}} |h^k-(h^k(\cdot,t))_{r'}|^{3/2}\,dx\,dt.$
\end{proof}

In fact, the above proof together with the estimates in Proposition~\ref{epsiloncriterion}
give the following  version of Lemma~\ref{stability}.
\begin{Lemma}
\label{stability2}
Under the assumptions of Proposition~\ref{compactness}, let $K$ be a compact subset of $\OO$.
If each point of $K$ is a regular point of $u$, then, for sufficiently large $k$, each point of
$K$ is also a regular point of $u^k$, and on the set $K$ the functions $u^k$ converge
to $u$, together with all spatial derivatives.
\end{Lemma}

\section{Mild solutions}
\label{sec:mildsolutions}
In this section we review the results we need about the so-called ``mild solutions"
of the problem~\ref{nse}--~\ref{inc}. This approach was introduced by Fujita and Kato, \cite{KatoH},
see also \cite{Kato}, although the terminology  was introduced later.
Let us first recall basic facts about the linear
Stokes problem
\begin{eqnarray}
\left.
\begin{array}{rcl}
u_t+\nabla p -\Delta u & = & \pd k f_k \label{linstokeseq}\\
\div u & = & 0
\end{array}
\right\}&  &
\mbox{in $R^n\times(0,\infty)$}\\
\begin{array}{ccl}
u(\,\cdot\,,0) & = & u_0\,\,\,\,\,\,\,\,\,\,\,\,
\end{array} & & \, \mbox{in $R^n$} \label{ic}
\end{eqnarray}

Here $f_k=(f_{1k},\dots,f_{nk})$ for $ k=1,\dots,n$.
Let $S(t)$ be the solution operator of the heat equation and let $P$ be the Helmholtz
projection of vector fields onto the divergence-free vector fields. By definition,
a mild solution of the linear problem above is given by the representation formula
\be
\label{repform1}
u(t)=S(t)u_0+\int_0^t S(t-s)P\nabla\cdot f(s)\,ds
\ee
A mild solution of the Cauchy problem~\ref{nse}--~\ref{inc} is the mild solution
of the linear problem above with $f_{ij}=-u_iu_j$.
We will denote the ``heat extension" $S(t)u_0$ of the initial datum $u_0$ by $U=U(x,t)$.
The term \hbox{$ \int_0^t S(t-s)P\nabla\cdot f(s)\,ds$} with $f_{ij}=-u_iv_j$ will be denoted
by $B(u,v)$.
In this notation, a mild solution of the Cauchy problem~\ref{nse}--~\ref{inc} in $R^3\times (0,T)$ is defined
as a solution of the integral equation
\be
\label{integraleq}
u=U+B(u,u)
\ee
in a suitably defined  space of functions $X$ on $R^3\times(0,T)$.
In this approach, a key property of $X$ is the continuity of the
bilinear form $(u,v) \to B(u,v)$ as a map from $X\times X$ to $X$.
This is equivalent to
\be
\label{Bcontinuity}
||B(u,v)||_X\le c||u||_X ||v||_X
\ee
For initial datum $u_0\in \hhd$ there are many possible choices
of $X$. A good choice is for example $X=\lthx 4 1$.
In this case the proof of~\ref{Bcontinuity} is particularly simple:
using the inequality $||fg||_{\hhd (R^3)}\le c||f||_{\hds 1 (R^3)}||g||_{\hds 1(R^3)}$
we see that for $u,v\in X$ we have $uv\in\lthx 2 {1/2}$.
Recalling the energy inequality for the linear system~\ref{linstokeseq} ,
\be
 \label{linenergy}
 ||\,|\nabla|^s u||^2_{\ltx \infty 2} +
 ||\,|\nabla|^s u||^2_{\lthx 2 1} \le
 ||\,|\nabla|^s u_0||^2_{L_x^2} +
 ||\,|\nabla|^s f|||^2_{\ltx 2 2}
\ee
one easily gets~\ref{Bcontinuity}.
Also, the energy inequality implies that $u_0\in \hhd$ gives $U\in X$, with
$||U||_X\le ||u_0||_{\hhd}$.

In fact, the above proof gives that $B$ maps $\lthx 4 1\times \lthx 4 1$ into
$\mathcal C_t\hhd_x\cap\lthx 2 {3/2}$ (where the first space denotes the space of continuous
functions of $t$ with values in $\hhd$), which shows that equation~\ref{integraleq}
can be treated as an ODE in $t$. In particular, one always has local-in-time
existence of the solution $u$, and one can define the maximal time of existence $\tmax(u_0)$
on which the solution of \ref{integraleq} exists.

If $\tmax(u_0)$ is finite, one has
\be
\label{limit}
\lim_{T\to \tmax(u_0)} ||u||_{\lthx 4 1 (R^3\times (0,T))}=\infty\,\,.
\ee

We note that for sufficiently small $||u_0||_{\hhd}$ we have $\tmax(u_0)=+\infty$.
This justifies the definition of $\rmax$ from the Introduction.

 The mild solutions have the same regularity as $U$ since, roughly speaking,
 for short times they are just a perturbation of $U$, and this can be
 iterated forward to the time interval where the solution exists.\footnote{
 One has to be cautious with this statement if ``regularity" also
 means decay properties of $u$ as $x\to\infty$. Due to the non-local
 effect of the constraint $\div u = 0$, the solutions $u$ can have slower
 decay at infinity than $U$. See for example \cite{Brandolese}.}
 In particular, the mild solutions are smooth in
 $R^3\times(0,\tmax(u_0))$.

 One obvious reason for $\tmax(u_0)$ to be finite would be the development
 of a singularity in the solution $u$ at time $\tmax(u_0)$.
 A-priori it is not clear that this is the only reason. One could also imagine
 a scenario where the $\lthx 4 1$ norm of the solution would blow up
 even though the solution would remain smooth on each compact subset.
 However, this scenario can be ruled out. The only reason for
 the blow-up of the  $\lthx 4 1$ norm of the mild solutions $u$ with
 $u_0\in\hhd$ are the possible finite-time
 singularities. This will be justified in the next section.

 \section{Leray's solutions}
 \label{sec:leray}
In his pioneering work \cite{Leray34} Leray proved the existence
of the weak solutions to the Cauchy problem \ref{nse}--~\ref{inc}.
A key ingredient in his approach is the energy inequality
\be
\label{globalenergy}
\int_{R^3} |u(x,t)|^2\,dx + 2\int_0^t\int_{R^3}|\nabla u(x,s)|^2 \,\,dx\,ds\le\int_{R^3} |u_0(x)|^2\,dx
\ee
This inequality is the only known a-priori bound for general solutions.
At the first glance it would seem that for its application it is crucial that
$u_0\in L^2(R^3)$, which would rule out using Leray's techniques in
the situation of the preceding section, where the basic assumption is $u_0\in\hhd$,
which is not a subset of $L^2$.

However, Lemari\'e-Rieusset \cite{Lemarie} found a generalization of~\ref{globalenergy}
to the situation when the energy is only (uniformly) locally finite, and this makes
it possible to extend the theory of Leray's weak solutions to a much more general setting.
See also \cite{KiSer}.
In this paper we will not need the full version of Lemari\'e-Rieusset's local
theory, but we will need a version of his inequality for local energy, see Lemma~\ref{lemarieestimates}.

In our setting with $u_0\in \hhd$ one can use the following trick by C.~Calderon \cite{CCalderon}
to construct the weak solutions in a simple way.
We can write $u_0=a_0+v_0$ with $a_0$ smooth and small in $\hhd$,
and $v_0$ in $L^2$. (For example, $a_0$ can be defined in terms of
the Fourier transform as $\hat a_0(\xi)=\hat{u}_0(\xi)\varphi(\xi)$,
where $\varphi$ is a suitable smooth function equal to $1$ in a small
neighborhood of $0$.) Since $a_0$ is small, the Cauchy problem~\ref{nse}--~\ref{inc}
has a global solution $a$ which is small in $\lthx 4 1$.
We now seek solutions $u$ in the form $u=a+v$, where $v$ is a new
unknown function satisfying the equation
\be
\label{vequation}
v_t+a\nabla v + v\nabla a + v\nabla v+\nabla q -\Delta v = 0 \,.
\ee
The energy identity for this equation is
\be
\int_{R^3}|v(x,t)|^2\,dx+2\izt\irth|\nabla v(x,s)|^2\,dx\,ds=\irth|v_0(x)|^2+2\izt\irth (a\nabla v)v\,.
\ee
The last integral on the right-hand side can be estimated by
$$c||a||_{\lthx 4 1}||v||^{1/2}_{\ltx\infty 2}||\nabla v||^{3/4}_{\ltx 2 2}$$
and we see that we have a good energy estimate for
$v$ when $||a||_{\lthx 4 1}$ is sufficiently small.

Leray's construction of weak solutions can therefore be applied to
equation~\ref{vequation} for $v$. This way we can construct a global
weak solution $u=a+v$ to the Cauchy problem with $u_0\in \hhd$.
The pressure can be recovered from the equation
$-\Delta p = \partial_i\partial_j u_iu_j$.
Moreover, one can do the construction in such a way that
 $(u,p)$ is a suitable weak solution in $R^3\times(0,\infty)$
 and $u(t)\to u_0$ in $L^2$ on every compact subset of $R^3$.
The weak solution $u$  with these properties will be called
{\it the Leray solution}.
The relation of the Leray solution and the mild solution
introduced in the previous section is clarified by the following
``weak-strong uniqueness" theorem. In the case $u_0\in L^2\cap \hds 1$
the theorem was proved by Leray. Leray's result was generalized
in various directions, see for example \cite{Serrin, vonWahl, Lemarie}.
We will use the following version which is a special case
of Theorem 33.2, p. 354 from Lemari\'e-Rieusset's book \cite{Lemarie}.

\begin{Theorem}
\label{uniqueness}
Let $u$ be a Leray solution of the initial value problem~\ref{nse}--~\ref{inc}
with $u_0\in\hhd$. Let $\tmax(u_0)$ be the maximal time of existence
of the mild solution with of~\ref{nse}--~\ref{inc} with the same initial value $u_0$.
Then the mild solution coincides with $u$ in $R^3\times[0,\tmax(u_0))$.
\end{Theorem}

  The problem of uniqueness of $u$
after $\tmax(u_0)$  is open. At the time of this writing we
cannot rule out that $\tmax(u_0)$ is finite and that the Leray
solution is not unique after $\tmax(u_0)$.
We will denote the set of all Leray solutions
with initial data $u_0\in \hhd$ by ${\rm NS}(u_0)$.

Proposition~\ref{epsiloncriterion} can be used to show that the only
reason for $\tmax(u_0)<\infty$ can be a finite time singularity.
We will now sketch a proof of this statement.
Let us assume that $T=\tmax(u_0)$ is finite. Set $r=\sqrt{T/2}$.
We consider the decomposition
$u=a+v$ as above, where $a$ is a solution generated by $a_0$ with
small $||a_0||_{\hhd}$ (and hence $a$ satisfies global estimates)
and $v$ satisfying the energy estimates.
The key point is that these estimates do not deteriorate as we approach $T$.
Using these estimates, together with corresponding estimates for the pressure,
it is not hard to see that for sufficiently large $R>0$, the assumptions
of Proposition~\ref{epsiloncriterion} are satisfied for our solution
$(u,p)$ and $Q_{z_0,r}$ with $z_0=(x_0,T)$ and $|x_0|>R$.
If $u$ does not develop a singularity at time $T$ in the ball $B_R$,
it means that $u$ and $\nabla u$ will be bounded in $(t_1,T)$ for any $t_1>0$.
We can now write the Navier-Stokes equation for $u$ as
\be
\label{linnse}
u_t-\Delta u + \nabla p = -\div (u\otimes u)\,\,.
\ee
Using the standard estimates for the small solution $a$, the energy estimates
for $v$, together with the pointwise bound for $u$ and $\nabla u$, one can
easily show that the term $u\otimes u=(a+v)\otimes(a+v)$ is
bounded both in $\ltx 2 2(R^3\times(T/2,T))$ and $\lthx 2 1(R^3\times(T/2,T))$,
and therefore also in $\lthx 2 {1/2} (R^3\times(T/2,T))$.
Viewing~\ref{linnse} as a linear equation with the right-hand side $-\div(u\otimes u)$,
we see by the energy estimate that $u\in\lthx 4 1 (R^3\times (0,T))$,
which means that $T$ was not the maximal time of existence of the mild solution,
a contradiction.

The weak solution $v$ of equation~\ref{vequation} always belongs to the energy space
$\ltx \infty 2 \cap \lthx 2 1$. As noticed already by Leray in \cite{Leray34},
this implies that $v$ is smooth and small for large times. In our set-up
we can see this from the fact that $||v(t)||^2_{\hhd}\le ||v(t)||_{L^2}||\nabla v(t)||_{L^2}$
and the expression on the right-hand side of this inequality clearly has to
be small on a large set of times if $v$ is in the energy space. Following the reasoning of Leray,
(\cite{Leray34}, p.\ 246), we notice that at a time $t_0$ when $||v(t_0)||_{\hhd}$
is small we can apply the theory of mild solutions and the weak-strong uniqueness
results to see that after time $t_0$ the solution $v$ coincides
with a global mild solutions. Similar considerations have been used for example in
 \cite{GallagherPlanchon}.

We can summarize the above facts in the following statement:
\begin{Proposition}
\label{compactnessofsing}
Let $u$ be a Leray solution of the Cauchy problem~\ref{nse}--~\ref{inc} with $u_0\in\hhd$.
Then for some compact set $K\subset R^3\times(0,\infty)$ we have
$\nabla u\in \ltx 4 2 (R^3\times(0,\infty)\setminus K)$.
In particular, $u$ is regular at every point of \hbox{$R^3\times(0,\infty)\setminus K$.}
\end{Proposition}

\begin{proof}
The proof of the first statement follows from the comments above. The second
statement follows from the first and the standard regularity criteria,
such as the Ladyzhenskaya-Prodi-Serrin criterion, or from Proposition~\ref{epsiloncriterion}.
\end{proof}

The energy estimate for $v$ which can be obtained from equation~\ref{vequation} depends
on the decomposition of the initial data $u_0=a_0+v_0$. For our purposes in this paper
we need an energy estimate which is ``more uniform" (although more local).
Fortuitously, an estimate found by Lemari\'e-Rieusset in his work on weak solutions
with locally finite energy provides exactly what we need.
We will use the following notation: for $x_0\in R^3$ and $r>0$ we will denote by
$\tilde Q_{x_0, r}$ the space-time cylinder $B_{x_0,r}\times(0,r^2)$.
We will also use the notation $||u||_{\EE(\tqr)}$ to denote the
energy norm defined by
\be
\label{energynorm}
||u||^2_{\EE(\tqr)}=||u||^2_{\ltx\infty 2(\tqr)}+||\nabla u||^2_{\ltx 2 2(\tqr)}.
\ee
\begin{Lemma}
\label{lemarieestimates}
Let $u_0\in\hhd$ and let $u$ be a Leray solution of the Cauchy problem~\ref{nse}--~\ref{inc}
with initial condition $u_0$.  Then for each $r > 0$ and $x_0\in R^3$
\be
\label{lemarieestimate1}
||u||^2_{\EE(\tqr)}\le C(||u_0||_{\hhd})\,\,r
\ee
and, for a suitable function $p_{x_0,r}(t)$ of $t$,
\be
\label{lemarieestimate2}
\dint_{\tqr}|p-p_{x_0,r}(t)|^{3/2}\,dx\,dt\le C(||u_0||_{\hhd})\,\,r^{2}
\ee
\end{Lemma}

\begin{proof}
The first estimate can be easily derived from Proposition 32.1, p.\ 342
and its proof in \cite{Lemarie}, and the second estimate follows from
Lemma 32.2, p.\ 343 in the same book. There are two crucial points
in the proof of these estimates. One is that the energy flux in the
localized energy estimate~\ref{localenergy} is bounded by $\sim |u|^3$,
if we count the pressure as $p\sim |u|^2$. The energy itself
controls $\sim |u|^{10/3}$, and it is important for the
proof that this is a higher power than the one in the energy flux.
This is no longer the case in higher dimensions and therefore a possible
generalization to higher dimensions would not be straightforward.
Similar issue arises in the proof of partial regularity.
The second point is that the non-local effects of the pressure
are under control, so that the heuristics $p\sim |u|^2$
is valid, at least as far as the estimates are concerned.
To see this, we notice that the kernel of the pressure equation
\be
\label{p}
-\Delta p = \partial_i\partial_j(F_{ij}), \qquad \mbox{with $F_{ij}=u_iu_j$}
\ee
is
\be
\label{kernel}
G_{ij}=\partial_i\partial_j G, \qquad \mbox{with $G(x)={1\over{4\pi|x|}}$}\,.
\ee
Therefore the kernel for expressing the gradient $\nabla p$ in terms
of $F_{ij}$ decays as $|x|^{-4}$ as $x\to\infty$, and  is
integrable near $\infty$. This fast decay makes it possible
to  estimate  the contributions to $\nabla p$
from far away.  This would not be the case for $p$, and hence
we have to work with $\nabla p$, which is the reason for the appearance
of the function $p_{x_0,r}$ in the estimate~\ref{lemarieestimate2}.
This part of the argument would work in the higher dimensions as well.
\end{proof}

We can now formulate the main new result of this section.
\begin{Theorem}
\label{main}
Let $u^k_0$ be a bounded sequence of initial conditions in $\hhd$ converging weakly
in $\hhd$ to $u_0$. Let $u_k\in {\rm NS}(u^k_0)$ be
Leray solutions of the Cauchy problem with initial conditions $u^k_0$.
Assume that $u^k$ converge weakly to $u$ in distributions.
Then $u\in{\rm NS}(u_0)$, i.\ e.\ $u$ is a Leray solution
of the Cauchy problem with initial condition $u_0$.
\end{Theorem}

\begin{proof}
Using Lemma~\ref{lemarieestimates}, Proposition~\ref{compactness} and Theorem~\ref{uniqueness},
we see that it is enough to show that $u(t)\to u_0$ in $L^2$ on every compact subset
of $R^3$. This can be seen as follows. We take a non-negative smooth test function
$\phi(x,t)$ compactly supported in $R^3\times[0,\infty)$. Note that we are taking
the interval $[0,\infty)$, which is closed at zero, and $\phi(x,0)$ does
not have to vanish everywhere. We write a version of the local energy inequality
with such test functions in the following form.
\begin{eqnarray}
\label{localenergy2}
\lefteqn{\int_0^\infty\int_{R^3}\left[ -|u^k|^2\phi_t+ 2|\nabla u^k|^2\phi\right]\,\,dx\,dt \le} \\
 & & \int_{R^3} |u^k_0|^2\phi(x,0)\,\,dx + \int_0^\infty\int_{R^3}\left[ |u^k|^2 \Delta \phi + (|u^k|^2+2p^k)u^k\nabla\phi\right]\,dx\,dt \nonumber
\end{eqnarray}

Since for every compactly supported smooth test function $\psi$ the sequence
 $u^k_0\psi$ is compact in $L^2$, we see that in the limit $k\to\infty$
 the inequality~\ref{localenergy2} will be preserved.
 Hence
 \begin{eqnarray}
\label{localenergy3}
\lefteqn{\int_0^\infty\int_{R^3}\left[ -|u|^2\phi_t+ 2|\nabla u|^2\phi\right]\,\,dx\,dt \le} \\
 & & \int_{R^3} |u_0|^2\phi(x,0)\,\,dx + \int_0^\infty\int_{R^3}\left[ |u|^2 \Delta \phi + (|u|^2+2p)u\nabla\phi\right]\,dx\,dt \,\,,\nonumber
\end{eqnarray}
where $p$ is a suitable pressure corresponding to the solution $u$.
The last inequality implies the required local $L^2$-continuity property at time $t=0$
for the solution $u$. The key points, well-known in the theory of weak solutions
and going back to Leray's paper \cite{Leray34} are that

\noindent
(a) the equation implies that
$u(t)\to u_0$ weakly in $L^2$ on compact subsets as $t\to 0$ and

\noindent
(b) inequality~\ref{localenergy3} implies that for every compactly supported smooth
test function $\psi$ one has  $\limsup_{t\to 0_+}||u(t)\psi||\le||u_0\psi||$.

\end{proof}

\begin{Corollary}
Let $u_0^k, \,u_0,\, u^k$ be as in Theorem~\ref{main}. Let $(0,\tmaxu)$ be the maximal
interval of existence of the mild solution $u$ starting at $u_0$.
Then for any compact set $K\subset R^3\times(0,\tmaxu)$ and any $k\ge k_0=k_0(K)$
the solutions $u^k$ are regular at all points of $K$ and converge uniformly
to $u$ in $K$, together with all spatial derivatives.
\end{Corollary}

\begin{proof}
Apply the theorem, together with Lemma~\ref{lemarieestimates}, Proposition~\ref{compactness} and Lemma~\ref{stability2}.
\end{proof}
\begin{Corollary}
\label{stabil}
Let $u_0^k, \, u_0,\, u^k$ be as in the theorem. Assume that $\tmax(u_0^k)=T<+\infty$
for each $k$ and that the sigular points $z_k$ of $u$ at $t=T$ (which exist by
Proposition~\ref{compactnessofsing}) stay in a compact subset of $R^3\times\{T\}$.
Then $\tmax(u_0)\le T$.
\end{Corollary}

\begin{proof}
Apply the theorem, together with Lemma~\ref{lemarieestimates}, Proposition~\ref{compactness} and Lemma~\ref{stability}.
\end{proof}

Let us recall that
\be
\label{rmaxdef}
\rmax=\sup\{\rho;\,\, \tmax(u_0)=+\infty \quad\mbox {for every $u_0\in\hhd$ with $||u_0||_{\hhd}<\rho$}\}
\ee
Let us also define
\be
\label{Mdef}
\MM=\{u_0\in \hhd;\,\,\tmax(u_0)<\infty,\, ||u_0||_{\hhd}=\rmax\}
\ee
\begin{Corollary}
\label{Mcompactness}
The set $\MM$ is non-empty. Moreover, $\MM$ is compact modulo scalings and translations,
 i.\ e.\ if $u^k_0\in\MM$ is a sequence in $\MM$, then there exist
$\ll_k>0$ and $x^k_0\in R^3$ such the sequence $v^k\in\hhd$ defined by
$v^k(x)=\ll_ku^k_0(\ll_k x-x_0^k)$ is compact in $\hhd$.
\end{Corollary}
\begin{proof}
Let $u^k_0\in\hhd$ be a sequence of initial data with $\tmax(u_0^k)$  finite and
$||u_0^k||_{\hhd}\to\rmax$.
Find $\ll_k>0$ and $x^k_0$ so that the functions given by
$v^k(x)=\ll_ku^k_0(\ll_k x-x_0^k)$ develop their first singularity at time $t=1$
and that $(x,t)=(0,1)$ is a singular point of $v^k$. We can assume that the functions
$v^k_0(x)=v^k(x,0)$ converge weakly in $\hhd$ to $v_0\in\hhd$.
By Corollary~\ref{stabil} we know that $\tmax(v_0)\le 1$, and
by definition of $\rmax$ this means that $||v_0||_{\hhd}=\rmax$
This shows that $\MM$ is non-empty.
We also see that $||v_0^k||\to ||v_0||$ and hence $v_0^k\to v_0$ strongly.
\end{proof}

The following corollary can be thought of as a variant of results
in \cite{Gallagher} and \cite{Tao}. See also Theorem 33.3, p.\ 359 in \cite{Lemarie}
for a related statement about ``individual solutions".
\begin{Corollary}
\label{cor:aprestimate}
Assume that every solution of the Cauchy problem~\ref{nse}--~\ref{inc} with
$u_0\in \hhd$ is regular, i.\ e. $\tmax(u_0)=+\infty$ for each
$u_0\in\hhd$. Then, for $l=0,1,2,\dots$ there exist functions $F_l\colon[0,\infty)\to[0,\infty)$
such that
\be
\label{aprestimate}
t^{(l+1)/2}\sup_x|\nabla^l u(x,t)|\le F_l(||u_0||_{\hhd})
\ee
\end{Corollary}

\begin{proof}
Let us prove the case $l=0$, the proof for the derivatives being the same.
By scaling invariance, it is enough to prove the statement for $t=1$.
If the statement fails we can assume by the translational invariance that there
exists a sequence of initial data $u_0^k$ bounded in $\hhd$ such that
for the corresponding solutions $u^k$ one has  $|u^k(0,1)|\ge k$.
Let $u_0$ be a weak limit of $u_0^k$, By our assumption, the solution
$u$ corresponding to $u_0$ is regular at $(x,t)=(0,1)$ and
by Theorem~\ref{main} and Lemma~\ref{stability2} we get a contradiction.
\end{proof}

\end{document}